\documentclass[10pt]{amsart}
\usepackage{graphicx}
 \usepackage[parfill]{parskip}
 \usepackage{amsmath}
\usepackage{amsthm}
\usepackage{amssymb}
\usepackage[T1]{fontenc}
 \newtheorem{theorem}{Theorem}
 
\newtheorem{lemma}[theorem]{Lemma}
\newtheorem{remark}[theorem]{Remark}

\newtheorem{corollary}[theorem]{Corollary}

\def\R{\Bbb R}

\def\F{{\mathbb F}}

\def\P{{\mathbb P}}

\title{On incidences of lines in regular complexes}

\author{Misha Rudnev}
\address{Misha Rudnev, Department of Mathematics, University of Bristol,
  Bristol BS8 1TW, United Kingdom}
\email{misharudnev@gmail.com}

\begin{document}
\begin{abstract} 
	A regular linear line complex is a  three-parameter set of lines in space, whose Pl\"ucker vectors lie in a hyperplane, which is not tangent to the Klein quadric. Our main result is a bound $O(n^{1/2}m^{3/4} + m+n)$ for the number of incidences between $n$ lines in a complex and $m$ points in $\F^3$, where $\F$ is a field,  and $n\leq char(\F)^{4/3}$ in positive characteristic.  Zahl has recently observed that bichromatic pairwise incidences of  lines coming from two distinct line complexes account for the nonzero single distance problem for a set of $n$ points in $\F^3$. This implied the new bound $O(n^{3/2})$ for the number of realisations of the distance, which is a square, for  $\F$, where $-1$ is not a square in the $\F$-analogue of the Erd\H os single distance problem in $\R^3$. Our incidence bound yields, under a natural constraint, a weaker bound $O(n^{1.6})$, which holds for any distance, including zero, over any $\F$. \end{abstract}
\maketitle

\section{Introduction}
In 1946 Erd\H os posed the problem of estimating how many pairs of points in a set of $n$ points could be at given distance from each other \cite{E}. By considering points in a square grid in $\R^2$ he conjectured the bound $O(n^{1+c/\log \log n}),$ for some $c>0$. The problem is wide open, with the best known bound $O(n^{4/3})$ being a direct consequence of (a generalised version of) the Szemer\'edi-Trotter theorem in $\R^2$ \cite{SST}. The single distance problem is naturally open in $\R^3$ as well, where the best known bound $O(n^{3/2-c})$, for $c<1/394,$ is due to Zahl \cite{Z}.

Since the early 2000s,  analogues of Erd\H os-type problems in arithmetic and geometric combinatorics have been actively explored in the finite field $\F$ setting, where they have so far proved to be more difficult, at least in the range when $n$ is relatively small relative to $p={\rm char}\,\F$. See, e.g. \cite{BKT}, \cite{IR}.  In this setting, nothing better than the trivial bound $O(n^{3/2})$ is known in two dimensions. 
However, in three dimensions, remarkably, Zahl \cite{Za} has recently proven the same bound $O(n^{3/2})$, under some constraints to be described shortly, roughly restricting his bound to ``50\%'' of the distances.
This note builds up on the new geometric insight in the latter paper of Zahl, aiming to extend its results, roughly to ``all distances''. However, it only succeeds in establishing a weaker bound 
$O(n^{8/5})$.

\medskip
Throughout $\F$ is a field of positive characteristic $p$. The claims are trivial for small $p$, so we assume that $p$ is odd. The results apply to the $p=0$ case  by just ignoring the constraints involving $p$. 


The $d$-dimensional projective space over $\F$ is denoted as $\P^d$, and $\P^3$ is often restricted to the chart $\F^3$. Furthermore, $\mathcal K\subset \P^5$ denotes the Klein quadric --  the space of lines in $\P^3$, see \cite[Chapter 6]{S} for foundations of line geometry in three dimensions.  In the Pl\"ucker coordinates\footnote{If a line $l$ in $\P^3$ is defined by two distinct points on it, with homogeneous coordinates $(x_0:x_1:x_2:x_3)$ and  $(y_0:y_1:y_2:y_3)$, then $P_{ij}:=x_iy_j-x_jy_i$, the point $\boldsymbol P\in \mathcal K$ being called the Klein image of line $l$ in $\P^3$.}
$$
\boldsymbol P := [P_{01}:P_{02}:P_{03}:P_{23}:P_{31}:P_{12}] $$
in $\P^5$, the equation of $\mathcal K$ is
$$
P_{01}P_{23}+ P_{02}P_{31}+ P_{01}P_{31}=0,\,\;\mbox{ equivalently  } \; \boldsymbol \omega\cdot \boldsymbol v = 0\,,
$$
using the notation $\boldsymbol P = [\boldsymbol \omega:\boldsymbol v]$ as to its first:last three components, where $\cdot$ is the standard dot product.
The lowercase  $l$ denotes a  {\em physical} line in $\P^3$,  whose {\em Klein image} is a point in $\mathcal K\subset \P^5$.  For lines not lying in the plane at infinity, the component $\boldsymbol \omega$ of their Pl\"ucker coordinates is their direction vector.

The standard notations $\ll,\gg$ are used to suppress absolute constants in inequalities, the symbol  $\sim$ means that both relations $\ll,\gg$ hold simultaneously, and the $O$ symbol is occasionally used instead of $\ll$.

Let $A\subset \F^3$ be a set of $n$ points, with $a=(x,y,z)\in A$ in the standard basis. Zahl \cite[Theorem 1.9]{Za} proved the following bound on the number of occurrences of the distance $r\neq 0$ between pairs of points of $A$.
\begin{theorem}\label{th:Zahl_1.9} Suppose, $-1$ is not a square in $\F$, $r\neq 0$ and $n\leq p^2$. Then 
$$
|\{(x_1,y_1,z_1),(x_2,y_2,z_2)\in A\times A:\, (x_1-x_2)^2+(y_1-y_2)^2+(z_1-z_2)^2= r^2\}|\ll n^{3/2}\,.
$$
\end{theorem}
He also points out that the claim of the theorem may not apply without additional constraints if $r^2$ gets replaced by $-r^2$ (provided that $-1$ is not a square in $\F$) or if $r=0$.
The reason for this is that for fixed $(x_1,y_1,z_1)$, the ``sphere'' 
$$
\{(x,y,z)\in \F^3:\, (x-x_1)^2+(y-y_1)^2+(z-z_1)^2= -r^2\}
$$
is ruled by isotropic lines\footnote{An isotropic line in $\F^3$ is a line, whose direction vector $\boldsymbol \omega$ is self-orthogonal, namely $\boldsymbol\omega\cdot \boldsymbol\omega=0$; there is a cone of such lines through every point in $\F^3$.}  (being doubly ruled for $r\neq0$ and a cone for $r=0$), see \cite[Lemma 6]{Ru1}. Thus for $r=0$, a set $A = \sqcup_{i=1}^k A_i$, with $k<c n^{1/2}$, where $c$ henceforth stands for a suitably small absolute positive  constant, and each $A_i$ of equal size being supported on an isotropic line would provide a counterexample to an extension of the claim of Theorem \ref{th:Zahl_1.9} without additional assumptions.  Similarly for $r\neq 0$, $A_i$ should be split half and half between a pair of parallel isotropic lines, one containing ``centres'' of the ``spheres'' and the other being contained in all of them. We further refer to spheres, their centres and radii (defined up to a sign choice) without quotation marks, including the $r=0$ case.

It appears that allowing $A$ to have at most $O(n^{1/2})$ points on an isotropic line should suffice for generalising the claim of Theorem \ref{th:Zahl_1.9} to the case of $-1$ being a square or generally $\F$ being algebraically closed. It turns out that such a generalisation is fully contingent on being able to handle the case $r=0$. In a sense, the obstruction due to many pairs of points of $A$ lying on isotropic lines  that Zahl points out is the only one. This is demonstrated in the proof of the forthcoming Corollary \ref{t:main} and in particular of Lemma \ref{l:con}, on which it relies.

We only succeed in proving a weaker bound $\ll n^{1.6}$ (rather than $\ll n^{1.5}$), generalising Theorem \ref{th:Zahl_1.9}, see  Corollary \ref{t:main}. It is closely related to our incidence bound between points and isotropic lines in  $\F^3$, Corollary \ref{c:iso}. The latter bound itself  arises as the dual statement to Theorem \ref{t:complex}, which is the main vehicle in this note. Theorem \ref{t:complex} is an incidence bound between a set of lines in a general linear regular line complex (to be defined shortly) and a set of points in $\F^3$.

The new geometric insight that enabled Zahl to arrive in the claim of Theorem \ref{th:Zahl_1.9} is as follows.  For $a=(x_1,y_1,z_1)$ and $b=(x_2,y_2,z_2)$, define the distance between them as
$$\|a-b\|:=(x_1-x_2)^2+(y_1-y_2)^2+(z_1-z_2)^2\,.$$

Then for any $r\in \F$, the condition 
\begin{equation}
\|a-b\|= r^2 = (r_1-r_2)^2\,
\label{distance}
\end{equation}
is equivalent, for any $r_1,r_2:\,r_1-r_2=r$, to the fact that two lines $l_1=l_1(a)$ and $l_2=l_2(b)$, given respectively, for $j=1,2,$ by Pl\"ucker vectors
\begin{equation}\label{complex}
\boldsymbol P =[1:-z_j+r_j :x_j-iy_j:r_j^2-x_j^2-y_j^2-z_j^2:-(z_j+r_j):x_j+iy_j]\,,
\end{equation} 
meet.  Observe that if $-1$ is not a square in $\F$, these lines live in the complexification $\F[i]^3$ of $\F^3$; otherwise they live in $\F^3$ itself.

This equivalence follows directly from the well-known fact that two lines $l,l'$ in $\P^3$ meet iff the Pl\"ucker vectors of their Klein images satisfy the equation
\begin{equation}\label{inter}
\boldsymbol \omega\cdot\boldsymbol v'+ \boldsymbol \omega'\cdot\boldsymbol v=0\,.
\end{equation}
With the usual formal dot product notation, can rewrite the latter condition in Pl\"ucker coordinates as
\begin{equation}\label{interinv}
\boldsymbol  P\cdot \mathcal I \boldsymbol P' = 0, \qquad \mbox{with} \qquad \mathcal I(\boldsymbol \omega:\boldsymbol v) := (\boldsymbol v:\boldsymbol \omega) \,,
\end{equation}
where $\mathcal I$ is an involution in $\P^5$.

Condition \eqref{distance} also represents an instance of {\em spherical tangency}, namely that two spheres, centred respectively at $a,b$ and with radii $r_j$, share a tangent plane at some point. Furthermore, \eqref{complex} represents a linear injection $\phi:\,(x_j,y_j,z_j,r_j)\to \mathcal K$ (quadraticity of the $P_{23}$-component in  \eqref{complex} being due to the Klein quadric equation). The instant of intersection of the lines, corresponding to $\phi(x_1,y_1,z_1,r_1)$ and $\phi(x_2,y_2,z_2,r_2)$ is equivalent to the distance condition \eqref{distance} concerning their preimages. A similar line of argument was pursued in  \cite{RS}, where the preimages, rather than spheres, i.e. quadruples $(x_j,y_j,z_j,r_j)$ here, were pairs of points in $\F^2$.  

We next fetch the concept of a {\em line complex} to describe the rationale behind the proof of Theorem \ref{th:Zahl_1.9} and develop its  generalisation herein.
A {\em linear  regular line complex} $C$ is the set of lines in $\P^3$, which are the Klein preimages of $\mathcal C:=\mathcal K\cap \mathcal H$, for a hyperplane $\mathcal H\subset\P^5$, whose normal vector $\boldsymbol N=[\boldsymbol n:\boldsymbol m]$ is not in $\mathcal K$. (Otherwise the linear line complex is called {\em singular}. Then the hyperplane $\mathcal H$ is tangent to $\mathcal K$ at the point $[\boldsymbol m:\boldsymbol n]$, and the Klein preimages of $\mathcal K\cap \mathcal H$ are  lines, meeting the line, which is the Klein preimage of the above tangency point.) If $\F$ is an algebraically closed field, all linear  regular line complexes, often referred to below just as {\em line complexes}, are projectively equivalent.  For the purposes of this paper we may assume that $\F[i]$, over which the line complexes in consideration are defined, is algebraically closed.  More about the theory of line complexes can be found in  \cite[Chapter 3]{PW}.

The hyperplane  $\mathcal H$ defining $\mathcal C$, the Klein image of the line complex $C$, cuts every two-plane inside $\mathcal K$ transversely. A plane inside $\mathcal K$ corresponds to physical lines, which are either concurrent or coplanar -- there are two rulings of $\mathcal K$ by planes. Thus $C$ is a three-dimensional family of lines, coming in pencils: lines in $C$, concurrent at any point in $q\in\P^3$ are automatically co-planar in some plane $\pi(q)$ (which arises by acting on $q$ with a $4\times 4$ nonsingular skew-symmetric matrix, whose entries can be read out from the equation of  $\mathcal H$, see \cite{PW}).

We now return to distances in $\F^3$, namely the discussion around relations \eqref{distance} and \eqref{complex}.
Observe that an instance of \eqref{distance} is satisfied for any $r$ and fixed $r_{1},r_{2}:\,r=r_1-r_2$ iff there is a {\em bichromatic} incidence between  a line from a family $L_1$ of $n$, say red lines in the complex $C_1=C_{r_1}$ and a line from a disjoint family $L_2$ of $n$, say  black lines in the complex $C_2=C_{r_2}$. The two line complexes and line families therein are defined, respectively, by  hyperplanes $\mathcal H_1$, $\mathcal H_2$ in $\P^5$, with the normal vectors
\begin{equation}\label{normal}
\boldsymbol N_j = [2r_j:-1:0:0:1:0]\,,\qquad j=1,2 \;\;\mbox{ and }\;\{(x_j,y_j,z_j)\} = A\,.
\end{equation}
Thus each $r\in \F$ defines a line complex $C_r$.

Once this has been set up, one enters the realm of incidence theory for lines in three dimensions, founded by Guth and Katz \cite{GK}. To generalise Theorem \ref{th:Zahl_1.9} to the case when $-1$ is a square in $\F$ we will use the forthcoming Theorem \ref{t:complex} together with the the state of the art incidence estimates, generalising the breakthrough \cite[Theorem 4.1]{GK}. For exposition purposes we split them into two statements. The first, bichromatic one is an amalgamation of  \cite[Theorem 12]{Ru}, \cite[Lemma 3.1]{FdZ}, \cite[Lemma 3.2]{Za}.

\begin{theorem}\label{th:KZ} Let $L_1\cup L_2$ be the union of $n\leq p^2$ red and $n$ black lines in $\F^3$, where at most $O(n)$ bichromatic pairs of lines lie in the same plane or doubly-ruled surface. 

Then the number of bichromatic pairwise line intersections is $O(n^{3/2})$.
\end{theorem}

The next theorem is due to Koll\'ar \cite[Theorem 2]{Ko}, see also \cite[Lemma 3.4]{Za}.
\begin{theorem} \label{th:Kol} Let $L$ be a family of $n\leq p^2$ lines in $\F^3$, where at most $O(n^{1/2})$ lines are coplanar, 
Let $Q_k$ be the set of points, where at least $k\geq 3$ lines meet. Then
\begin{equation}\label{e:gk}
|Q_k|\ll \frac{n^{3/2}}{k^{3/2}}\,.
\end{equation}
\end{theorem}

It turns out, see in particular the forthcoming Lemma \ref{l:con}, that if we add the condition that at most $n^{1/2}$ points of $A$ lie on an isotropic line in $\F^3$, the conditions of Theorems \ref{th:KZ}, \ref{th:Kol} are satisfied apropos of the line family $L_1\cup L_2$. Moreover, the maximum number of concurrent (hence coplanar) lines in a complex $C=C_r$ turns out to be be equal to the maximum number of points of $A$ on an isotropic line. Besides, the complimentary ruling condition of Theorem \ref{th:KZ} is satisfied quite generally, see Lemma \ref{l:skewn} below.

Theorem \ref{th:Zahl_1.9}, however, holds without the constraint on the maximum number of points on an isotropic line, because if $i\not\in \F$ and $r,r_1,r_2\neq 0$, then it turns out to be impossible that more than two lines of both colours are concurrent  or coplanar in $\F[i]^3$. Then one arrives in Theorem \ref{th:Zahl_1.9} by multiplying estimate \eqref{e:gk} by $k$ and summing over dyadic values of $k$ for $k=O(n^{1/2}),$ since concurrency of lines in a line complex automatically means coplanarity, thus for $k \geq 2n^{1/2}$ there is a better than \eqref{e:gk}  inclusion-exclusion bound $|Q_k|\ll n/k$.
 
 Here we analyse concurrency/coplanarity in a slightly different, more mindless way versus \cite{Za}\footnote{Zahl observes, remarkably, that if $i\not\in \F$, then for all $r$, the lines in the complex $C_r$ live in the Heisenberg surface $\Im z = \Im(x\bar y)$ in $\F[i]^3$, which has a natural Lie group structure.
This enabled him to give further consideration to the case $\F=\mathbb R$ by fetching polynomial partitioning techniques and using the line incidence estimate underlying Theorem  \ref{th:Zahl_1.9} as a base one to prove a generally sharp bound on the number of {\em rich pencils of contacting spheres}, see \cite[Theorem 1.14]{Za}. }  via direct calculations in the proof of Lemma \ref{l:con}.  
However, without the constraints of Theorem \ref{th:Zahl_1.9}, the corresponding multiple bichromatic line concurrencies may occur. This constitutes the principal obstruction against generalising the estimate of Theorem \ref{th:Zahl_1.9} to the case $r=0$ or when $-1$ is a square in $\F$.
Then, under a trivially unavoidable additional assumption that $A$ have no more than $n^{1/2}$ points on an isotropic line,  Theorems \ref{th:KZ}, \ref{th:Kol} alone would yield $O(n^{1.75})$ occurrences of a distance, which is worse than the bound $O(n^{1.6})$  that we claim. Our bound owes itself to the forthcoming Theorem \ref{t:complex} -- a new incidence bound that we interpolate with the bound of  Theorem \ref{th:Kol}.  It is still the multiple line concurrency issue that prevents us from vindicating  the estimate of Theorem  \ref{th:Zahl_1.9} when $\F[i]=\F$ or $r=0$, as our divide-and-concur approach effectively forces us to use estimate \eqref{e:gk}, with a fairly large threshold value of $k\sim n^{1/5}$. 


\medskip
We now state the main results of this note. The vehicle behind them is the following incidence bound.
\begin{theorem}\label{t:complex} Let $\F$ be a field of characteristic $p$. The number of incidences between $n\leq p^{4/3}$ lines from a linear regular line complex and  $m$ points in $\F^3$ is $O(m^{3/4}n^{1/2}+m+n ).$
\end{theorem}
The next two corollaries  follow from applying Theorem \ref{t:complex} over the field $\F[i]$.

One follows readily from Theorem \ref{t:complex} applied to the line complex $C_0$, defined by \eqref{complex}, with $r=0$. Observe that the lines in $C_0$ are themselves non-isotropic. 

\begin{corollary} \label{c:iso} The number of incidences between $m$ isotropic lines and $n\leq p^{4/3}$ points in $\F^3$ is $O(m^{3/4}n^{1/2} + n+m).$\end{corollary}
The other corollary is a weaker generalisation of Theorem \ref{th:Zahl_1.9}.
\begin{corollary}\label{t:main} Let $A\subset \F^3$ be a set of $n$ points with $n\leq p^{4/3}$ and at most $n^{1/2}$ points on an isotropic line. Then for any field $\F$ of characteristic $p$ and any $r\in \F$, $$
|\{(x_1,y_1,z_1),(x_2,y_2,z_2)\in A\times A:\, (x_1-x_2)^2+(y_1-y_2)^2+(z_1-z_2)^2= r^2\}|\ll n^{8/5}\,.
$$
\end{corollary}
\begin{remark} We remark that if one swaps $m$ and $n$ in the bound of Theorem \ref{t:complex}, the resulting incidence bound will replicate the Guth-Katz incidence theorem in $\R^3$ \cite{GK}. This would mean improving $k^{-3/2}$ to $k^{-2}$ in formula \eqref{e:gk} and claiming that families of lines in $\F^3$ that lie in a linear line complex satisfy the same pairwise intersection bounds as lines in $\R^3$. Such a bound 
would improve the bound in Corollary \ref{t:main} to $O(n^{3/2}\log n)$, replicating the claim of Theorem \ref{th:Zahl_1.9} modulo the factor $\log n$. Unfortunately, the author does not see how this swap of $m$ and $n$ can be vindicated, since the spaces of lines in a regular linear  line complex and $\P^3$ are not dual to each other.  Hence, this would in the least require additional geometric considerations, currently unavailable. Further indirect evidence to the heuristics that swapping $m$ and $n$ in the bound of Theorem \ref{t:complex} may be problematic  is furnished by the statement of Corollary  \ref{c:iso}, which  does effect the swap apropos of  isotropic lines and points. However, isotropic lines in $\F^3$ form a quadratic, rather than linear line complex, and its geometric structure is different. \end{remark}

\section{Proofs}

We start with a simple lemma that the author has not seen in the subject literature, and which might be of certain general interest. 
Without loss of generality as to the claim of the lemma, suppose $\F$ is algebraically closed. Call  a {\em regulus} (according to \cite{PW}, \cite{S}) a set of lines in $\P^3$, which are Klein preimages of a transverse intersection of $\mathcal K$ with a two-plane $\mathcal P$ in $\P^5$. It is easy to show (see the latter references) that these lines rule a quadric surface in $\P^3$, which is also ruled by the {\em complimentary regulus}.

\begin{lemma} 
For a doubly-ruled quadric surface in $\P^3$, the two-planes $\mathcal P_1,\,P_2,$ defining the two complimentary reguli, are mutually skew in $\P^5$. Hence, one of the rulings of the surface will have at most two lines, lying in a linear line complex.
\label{l:skew}
\end{lemma} 
\begin{proof} The intersection $\mathcal P_1\cap\mathcal K$ is a conic curve in $\mathcal P_1$. Take three distinct points on this conic: their Klein preimages $l_{1},l_2,l_3$ are mutually skew lines in $\P^3$. The complimentary regulus is the set of all lines, incident to all three of the above lines. That is $\mathcal P_2$ is the intersection of the three tangent hyperplanes to $\mathcal K$ at the Klein images of $l_{1},l_2,l_3$. If $\mathcal P_1\cap \mathcal P_2$ contains a point $x$, then $x\in \mathcal P_1$ is the intersection of three tangent lines to the conic curve in $\mathcal P_1$, which is a contradiction.

It follows that if a hyperplane $\mathcal H$ in $\P^5$ defines a line complex, then at least one of the two conics, corresponding to the two rulings of a quadric surface in $\P^3$ will meet it transversely. Hence, one of the rulings of the quadric surface has at most two lines, lying in the line complex. \end{proof}

We also present a statement, specific of the line complexes $C_1,\,C_2$, whose normals  are given by \eqref{normal}. The next lemma is  used in the proof of Corollary \ref{t:main}, and we present it as a somewhat alternative way to  \cite[Lemma 2.2]{Za}, to see that  the doubly-ruled surface condition of Theorem \ref{th:KZ} is satisfied.

\begin{lemma} 
If one ruling of a doubly-ruled quadric surface in $\P^3$ contains three lines from one of the complexes $C_1, \,C_2$, the complimentary ruling contains at most two lines from the other complex.
\label{l:skewn}
\end{lemma} 
\begin{proof} Let $l_1,\,l_2,\,l_3$ be three lines from, say $C_1$, lying in one ruling, with Pl\"ucker vectors, respectively  $\boldsymbol P_1,\, \boldsymbol P_2,\,\boldsymbol P_3$. Since $\boldsymbol N_1$ is orthogonal (in the dot product set) to all three Pl\"ucker vectors, the point $\mathcal I(\boldsymbol N_1) \in \mathbb P^5$ lies in the two-plane, defining the complimentary regulus (see \eqref{interinv} for notation). Since  $\boldsymbol N_2\cdot \mathcal I\,\boldsymbol N_1\neq 0$, this plane meets the hyperplane $\mathcal H_2$, defining the complex $C_2$, along a line, which is not contained in $\mathcal K$, or the complimentary regulus (represented by a conic in $\mathcal K$) would be a line pencil (a line in $\mathcal K$).  Thus, there are at most two lines in  $C_2$, lying in the complimentary regulus, they are the Klein pre-images of the intersection points of the latter line in $\mathbb P^5$ with $\mathcal K$. 
\end{proof}

\begin{proof}[Proof of Theorem \ref{t:complex}] Here $\F$ and $\P$ are not necessarily what they were at the outset: without loss of generality one may assume that $\F$ is algebraically closed. 

The main idea behind the formal proof that begins in the next paragraph is as follows. Central to it is an application of the point-plane incidence bound from \cite{Ru}, whose version is stated as the forthcoming Lemma \ref{ppb}. The bound may be used to control the number of pairwise intersections of $\leq p^2$ lines {\em inside} the Klein image $\mathcal C$ of a line complex (for another application of this kind see \cite[Lemma 10]{Ru}). Let us refer to lines in $\mathcal C$ as {\em phase lines} -- these are not {\em physical} lines in $\P^3$. Rather, they correspond to physical line pencils in $\P^3$. A  pairwise incidence of phase lines means the two physical line  pencils share a physical line. Now one views each phase line as the unique point-plane pair, defining the corresponding  physical line pencil. Hence a point-plane incidence bound can be used to control the number of incidences between a set physical lines and {\em pairs of line pencils}. Namely if a line $l$ belongs to two distinct line pencils $(q_1,\pi_1)$ (where $q_1$ is a point in $\P^3$, contained in a plane $\pi_1$) and $(q_2,\pi_2)$, then $q_1\in \pi_2$. Conversely, if $q_1\in \pi_2$, since also $q_1\in\pi_1$, then $l=\pi_1\cap \pi_2$ and $q_1,q_2\in l$, so the line $l$ lies in both pencils.

Thus let $L$ be the set of $n$ lines in a linear regular line complex $C$. For $2\leq k\leq n$ let $Q_k$ denote the set of {\em $k$-rich} points in $\P^3$, where some number between $k$ and $2k$ lines from $L$ meet (so $k$ may be assumed to take dyadic values $k=2^j$). Let us first observe that points, incident to one line only contribute $m$ to the total number of incidences. Moreover, we can assume that $k\leq k_0\sim\sqrt{n}$, for otherwise, by inclusion-exclusion principle, ``rich points'', to each of which  more than $k_0$  lines are incident, contribute $O(n)$ to the total number of incidences. See, e.g. \cite[Proposition 2.2]{Gu} for the standard proof.

Since the lines are in the line complex $C$, a point $q\in Q_k$ corresponds to a line pencil $(q,\pi)$, where $\pi$ is the plane, containing  all the lines from $L$ that meet at $q$.  For each line $l \in L,$ indexed, with a slight abuse of notation, by $l=1,\ldots,n$, let $t_l$ be the number of distinct points of $Q_k$ on this line.

By double counting, 
\begin{equation} k|Q_k| \sim \sum_{l=1}^n t_l\,.\label{e:out}\end{equation}  Let $L^*\subset L$ stand for ``rich lines'', that is lines $l:\,t_l> c^{-1}\sqrt{|Q_k|}$. The constant $c$ can be chosen to ensure that the usual inclusion-exclusion estimate yields
\begin{equation}\label{e:ii}
\sum_{l:\,t_l> c^{-1}\sqrt{Q_k}} t_l \leq |Q_k|\,.
\end{equation}
Thus, since $k\geq 2$, and applying Cauchy-Schwarz, 
$$
k|Q_k|\ll  \sum_{l:\, t_l\leq  c^{-1} \sqrt{|Q_k|}} t_l \leq   \sqrt{n} \sqrt{I}\,,
$$
with the term 
$$I:= \sum_{l:\, t_l\leq  c^{-1}\sqrt{|Q_k|}} t_l^2$$ summing, over lines in $L\setminus L^*$ the number of pairs of points of $Q_k$ on each line, that is the number of pairs of line pencils each line belongs to. Separating the count of pairs of distinct pencils for each line, we have

\begin{equation} 
k|Q_k|\ll  \sqrt{n k|Q_k| }  +  \sqrt{n} \sqrt{I_*}\,,
\label{e:dc} \end{equation}
where
$$
I_*:= \sum_{l\in L\setminus L^*} |\{(q,q')\in Q_k\times Q_k, \,q\neq q':\,q,q'\in l\} |\,.
$$

We next fetch the point-plane theorem to bound  $I_*$, as described at the outset of the proof. We quote  a special case \cite[Theorem 3*]{Ru} of the general bound \cite[Theorem 3]{Ru}, disregarding incidences, supported on the set of rich  lines $L^*$.

\begin{lemma}\label{ppb} Let $L^*$ be a set of lines in $\P^3$. For an arrangement $(Q,\Pi)$ of equal number of points and planes in $\P^3$, with $|Q|\leq p^2$, the number of incidences, not supported on $L^*$, is bounded as follows:
$$
|\{(q,\pi)\in Q\times \Pi:\, q\in \pi,\, q \mbox{ and } \pi \mbox{ not incident to a line in }L^*\}|\ll |Q|^{3/2} + K|Q|,
$$
where $K$ is the maximum number of points of $Q$, supported on a line not in $L^*$.
\end{lemma}

We apply Lemma \ref{ppb} to the point-plane arrangement of $|Q_k|$ points and planes, naturally corresponding to the set $Q_k$. In view of Lemma \ref{l:skew}, to  satisfy the $p$-constraint of Lemma \ref{ppb} we can use the estimate $|Q_k|\leq n^{3/2}$ of Theorems \ref{th:KZ} (formally with $L_1=L_2$), and Theorem \ref{th:Kol}\footnote{Since the lines lie in a single regular line complex, the coplanarity constraint of the theorems is irrelevant, since each of the $\ll n^{1/2}$ putative ``rich planes'' with more than $n^{1/2}$ lines therein will contribute at most one point to $Q_k$, corresponding to the line pencil in this plane. In addition, each line meets  $\ll n^{1/2}$ such planes. Thus lines lying in rich planes contribute $O(n^{3/2})$ to $|Q_k|$ and can be erased.}. Furthermore, by definition of the set of rich lines $L^*$, whose contribution is controlled by \eqref{e:ii}, the quantity $K$ in the application of the lemma is $O(|Q_k|^{1/2})$.

Hence, 
$$I_*=O(|Q_k|^{3/2})\,.$$ 
Then it follows from \eqref{e:dc} that  \begin{equation}\label{new} 
|Q_k|\ll \frac{n^2}{k^4}+\frac{n}{k}\,.
\end{equation} 
This essentially completes the proof of Theorem \ref{t:complex}, but for a technical issue that the two terms in the right-hand side of \eqref{new} meet when $k\sim n^{1/3}$, rather than $k\sim n^{1/2}$. This, without extra care, would cause $n\log n$ to replace the $n$-term in the estimate of Theorem \ref{t:complex}.

To avoid this, we bootstrap estimate \eqref{new} as follows. Return to estimate \eqref{e:out} and sum over dyadic values of $k=2^j,\,j\in J$, where $J:=\{j:\,cn^{1/3}\leq 2^j \leq c^{-1}n^{1/2}\}$.
It follows from \eqref{new} that
\begin{equation}\label{auxi} |Q_J|\ll n^{2/3}\,.\end{equation}

Let $X$ denote the left-hand side of estimate \eqref{e:out}, which we aim to bound as $X=O(n)$, namely 
$$
X:= \sum_{ k=2^j,\,j\in J} k|Q_k| \sim\sum_{l=1}^n \bar t_l\,,
$$
where $\bar t_l$ is now the total number of points of $Q_J:=\cup_{j\in J} Q_{k=2^j}$ on the line $l$. 

Repeating the argument leading from \eqref{e:out} to \eqref{e:dc}, now apropos of the quantity $X$, yields
$$
X\ll \sqrt{nX} + \sqrt{n} |Q_J|^{3/4}\,.
$$
It follows from \eqref{auxi} that $X\ll n$.

Finally, the main term $n^{1/2}m^{3/4}$ in the estimate of Theorem \ref{t:complex} arises from estimate \eqref{new} for $k\ll n^{1/3}$ in the standard way. As has already been mentioned, the cases $k=1$ and $k\gg n^{1/3}$ contribute $O(m+n)$ incidences.
For $k\ll n^{1/3}$ redefine  a $k$-rich point as the one where at least $k$ lines meet: dyadic summation for $k\ll n^{1/3}$ will have no effect on the dominating first term in \eqref{new}. Inverting estimate \eqref{new} (dropping the term $\frac{n}{k}$) as $k\ll n^{1/2} |Q_k|^{-1/4}$ means ordering points of pairwise intersections of lines from $L$ by non-increasing richness, so that the $t$th point on the list is incident to  $O(n^{1/2}t^{-1/4})$ lines. The claim of Theorem \ref{t:complex} follows after summing in $t\in [m]$.\end{proof}
 
 \begin{proof}[Proof of Corollary \ref{c:iso}]
Let $A\subset \F^3$ be the set of $n$ points in question and $L=L(A)$ the line family in the complex $C_0$, defined by \eqref{complex}, with $r=0$. A pair of distinct points of $A$ is at a zero distance, i.e. lies on an isotropic line in $\F^3$, iff two lines in $L(A)$ meet. The statement follows. \end{proof}

\begin{proof}[Proof of Corollary \ref{t:main}]
We need to bound from above the number of bichromatic line intersections from line families $L_1$, $L_2$, lying in two distinct line complexes $C_1=C_{r_1}$, $C_2=C_{r_2}$ for $r_1,r_2\neq 0$, defined by \eqref{complex},  or monochromatic intersections within the same line complex $C_0$ if $r=0$.

We use the following lemma, whose proof is a calculation, presented separately.

\begin{lemma}\label{l:con} Let $C=C_r$ be a line complex, defined by \eqref{complex}, for some $r\in \F$. Let $\bar \F=\F[i]$ and $\bar \P^3$ the three-dimensional projective space over $\bar \F$. If a set of lines from $C$ is concurrent at a point $\boldsymbol u\in \bar \P^3$, then the points of $A$, defining these lines lie on an isotropic line in $\F^3$. 

Moreover, if the complexes $C_1,\,C_2$ correspond to distinct $r_1,r_2\neq 0$ and  $\F\neq \bar \F$,  it is impossible that two lines of each colour be concurrent (coplanar) at the same point (plane) in $\bar\P^3$.
\end{lemma}

As was mentioned in the introduction, the claim of the lemma for $i\not \in \F$, leads to the statement of Theorem \ref{th:Zahl_1.9}, which holds without constraining the number of points of $A$ lying on an isotropic line.

If $i\in \F$, to prove Corollary \ref{t:main} we use Theorem \ref{t:complex} together with Theorems \ref{th:KZ}, \ref{th:Kol},  applied to the lines in $L_1\cup L_2$. Optimising the bounds \eqref{new} and \eqref{e:gk}, multiplied by $k^2$, in the range $2\leq k\leq n^{1/2}$ at $k=n^{1/5}$ completes the proof,  once the use of \eqref{e:gk} has been fully justified, including the case $k=2$.

For this purpose observe that the coplanarity condition of Theorems \ref{th:KZ}, \ref{th:Kol} is implied by  Lemma \ref{l:con} and the assumption that $A$ has at most $n^{1/2}$ points on an isotropic line in $\F^3$.
The complimentary ruling condition of Theorem \ref{th:KZ} follows from Lemma \ref{l:skewn}.
Thus, Theorems \ref{th:KZ} and \ref{th:Kol} both apply.
\end{proof}

\begin{proof}[Proof of Lemma \ref{l:con}]
The proof is a calculation, similar to those in \cite[Proofs of Theorems 3--5]{RS}. 
A line in the family $L=L(A)$ in the complex $C=C_r$, for $a=(x,y,z)\in A$, is given by the Pl\"ucker vector 
$$
[\boldsymbol \omega:\boldsymbol v] = [1:r-z:x-iy:r^2-\|a\|^2:-(r+z):x+iy]\,.
$$
Let us first check concurrencies at infinity by fixing $r-z = \omega_2$ and $x-iy = \omega_3$. It follows that if two or more points $a\in A$, satisfy this, they lie on an isotropic line with the direction vector $(1,-i,0)$, which is not in $\F^3$, unless $-1$ is a square in $\F$.

Otherwise, suppose concurrency happens at some point $\boldsymbol  u =(u_1,u_2,u_3) \in \bar\F^3$. Then one must have $\boldsymbol v = \boldsymbol u \times \boldsymbol \omega$, where $\times$ is the cross product. This yields two independent equations
$$
u_1(r-z) - u_2 = x+iy\,, \qquad u_3-u_1(x-iy) =-(r+z)\,.
$$
If $u_1=0$, this means $x+iy=u_2$ and $r+z=-u_3$, so if two or more points $a\in A$ satisfy this, they lie on an isotropic line with the direction vector $(1,i,0)$, which is not in $\F^3$, unless $-1$ is a square.

If $u_1\neq 0$, one gets, with $\alpha= u_1^{-1}-u_1$ and $\beta = u_1^{-1}+u_1$,
\begin{equation}\label{e:eqs1}
2x = \alpha z + \beta r + (u_3u_1^{-1}-u_2),\qquad 2y = i[ \beta z - \alpha r + (u_3u_1^{-1}+u_2)]\,.
\end{equation}
It follows that if there are two or more points $a\in A$, satisfying this, they all lie on some line, whose direction vector is $(\alpha, i \beta, 2)$, which is easily verified to be isotropic by the above definitions of the quantities $\alpha,\beta$.

If $i\not \in \F$, then in order for this direction vector to be in $\F^3$, one must have $\alpha,\,i\beta \in \F$. (These conditions are not independent and if $u_1=u_1'+iu_1''$, with $u_1',u_1''\in \F$, this happens for $u_1'^2+u_1''^2=-1$.)
Thus  for $i\not \in \F$ and $\boldsymbol u$ fixed, one cannot have two lines from $C_1$ and two lines from $C_2$ meet at $\boldsymbol u$. Indeed, since the values $r_1$, $r_2$, identifying the line complexes $C_1,\,C_2$ are nonzero, say the quantity $u_3u_1^{-1}-u_2$ in the first equation in \eqref{e:eqs1} has to be such that $\beta r + (u_3u_1^{-1}-u_2) \in \F$, which is impossible for two distinct nonzero values of $r$. 

Coplanarity analysis is similar. Since the Pl\"ucker coordinate $P_{01}=1$ for all lines, none one of them lies in the plane at infinity, and coplanarity of two lines means that for some $\boldsymbol  u =(u_1,u_2,u_3) \in \bar\F^3$ one has $\boldsymbol\omega = \boldsymbol u\times \boldsymbol v$.

This yields two independent equations
$$
u_2(x+iy)+u_3(z+r) =1\,, \qquad -u_1(x+iy) + u_3( r^2-\|a\|^2) = r-z\,.
$$
If $u_3=0$, then  if two or more points $a\in A$ satisfy this, they lie on an isotropic line with the direction vector $(1,i,0)$, which is not in $\F^3$, unless $-1$ is a square in $\F$. Similarly, if $u_2=0$, the same conclusion is drawn as to an  isotropic line with the direction vector $(1,-i,0)$.

If $u_2,u_3\neq 0$, one gets, with $w= u_3u_2^{-1}$, $\alpha= w^{-1}-w$ and $\beta = w^{-1}+w$, the following analogue of \eqref{e:eqs1}:
$$
2x = \alpha z - \beta r + (u_2^{-1}-u_1u_3^{-1}),\qquad 2y = i[ \beta z + \alpha r -  (u_2^{-1}+u_1u_3^{-1})]\,.
$$
It follows that if there are two or more points $a\in A$, satisfying this, they all lie on some line with isotropic direction vector $(\alpha, i \beta, 2)$. The rest of the conclusions apropos of coplanarity in the special case $i\not\in \F$ are the same as above.
\end{proof}

\paragraph{\textbf{Acknowledgments}.}  The author is partially supported by the Leverhulme Trust  Grant RPG--2017--371 and grateful to Joshua Zahl for sharing valuable insights, as well as the efforts of anonymous referees that have helped vastly to improve the presentation of the results in this paper.

\end{document}